\documentclass[12pt,reqno]{amsart}
 
 \usepackage[english]{babel}
 \usepackage[utf8]{inputenc} 
 \usepackage[T1]{fontenc} 
 \usepackage{textcomp}
 \usepackage{bm}
 \usepackage{cite}

\usepackage[margin=1.2in, marginparwidth=0.85in]{geometry}

\usepackage{mathtools,mathrsfs}
\mathtoolsset{showonlyrefs}

\usepackage{graphicx,enumerate}

\usepackage{hyperref,color}
\usepackage[dvipsnames]{xcolor}
\hypersetup{
	colorlinks=true,
	pdfpagemode=UseNone,
    citecolor=OliveGreen,
    linkcolor=RoyalBlue,
    urlcolor=black,
	pdfstartview=FitW
}
 
\usepackage{subcaption}
\usepackage{comment}

\usepackage{appendix}
\usepackage{multicol} 

\usepackage[tt=false]{libertine}
\usepackage[varqu,varl]{zi4}

\usepackage[amsthm]{libertinust1math} 

\usepackage{microtype}
\usepackage{graphicx}

\def\pr{\mathbb{P}}
\def\E{\mathbb{E}}

\newtheorem{thm}{Theorem}[section]
\newtheorem*{thm*}{Theorem}
\newtheorem{prop}[thm]{Proposition}
\newtheorem{lemm}[thm]{Lemma}

\theoremstyle{definition}
\numberwithin{equation}{section}

\newtheorem{defn}[thm]{Definition}

\title{Sample Path Large Deviations for Random Walks on Regular Trees}

\author{Jie Jiang}
\address{Jie Jiang: School of Statistics and Data Science, Nankai University, 94 Weijin Road, Nankai District, Tianjin, P.R. China} \email{2308053409(at)qq.com}

\author{Shuwen Lai}
\address{Shuwen Lai: School of Statistics and Data Science, Nankai University, 94 Weijin Road, Nankai District, Tianjin, P.R. China} \email{1120220052(at)mail.nankai.edu.cn}

\begin{document}

\begin{abstract}
This paper investigates the large deviation problem in the sample path space of the nearest-neighbor random walks on regular trees. 
 We establish the sample path large deviation principle for the law of the distance from a nearest random walk on a regular tree to the root with a good convex rate function. Furthermore, we derive an implicit expression for the rate function via the Fenchel-Legendre transform of the log-moment generating function.
\end{abstract}

\maketitle

\section{Introduction and Main Results}
Random walks (RWs) on graphs and groups have been one of the fascinating topics in probability theory over the past few decades. 
This field was initiated by Pólya's fascinating work \cite{Pol21} on the recurrence and transience of simple random walks on Euclidean lattices. 
Following this breakthrough, the mid-to-late 20th century witnessed an explosive development of RW theory across multiple dimensions including
\begin{itemize}
	\item Asymptotic properties: Quantitative analysis of speed, entropy and heat kernel decay;
	\item Structural connections: Interactions with geometric group theory and ergodic theory;
	\item Boundary theory: Compactification methods for RW convergence;
\end{itemize}
see \cite{Ave76,KV83,Var85,Lyo90,Gro92} and references therein.
For foundational learning on RWs, we refer readers to Woess' \cite{Woe00} and Lalley's \cite{Lal23}.

In this paper, we focus exclusively on random walks on regular trees, the Cayley graphs of the free product of two-element groups, and specifically on the study of their sample paths. The methods developed apply also, with some modifications, to random walks on free products of finitely generated groups.

Let $T_d$ be a $d$-regular tree with generating set $S=\{a_1,\ldots, a_d\}$ and root $e$, $d \geq 3$.  Let $\mu$ be a probability measure on $S$ and $ (Y_n)_{n\geq 0}$ a random walk on $T_d$ starting from $e$ with step distribution $\mu$, i.e., $Y_0=e$ and $Y_n= X_1 \cdots X_n$ for every $n\geq 1$, where the $X_n$ are independent $T_d$-valued random variables identically distributed according to $\mu$. 
Denote by $l$ the length function on $T_d$ determined by $S$, i.e., for $g \in T_d$, $l(g)$ is the graph distance from $g$ to $e$.

Using the Kingman's subadditive ergodic theorem \cite{Kin73}, we can see that  $(Y_n)_{n\geq 0}$ has a positive escape rate $c$. Specifically, the following strong law of large numbers (SLLN) holds:
\begin{equation*}
	\lim_{n\to\infty} \frac{l(Y_n)}{n} =\lim_{n\to\infty} \frac{\E[l(Y_n)]}{n} =c>0 \quad  \mathrm{a.s.}
\end{equation*}
In fact, Guivarc'h \cite{Gui80} proved a more general result: under the condition that step distribution has finite first moment with respect to the length function, the random walk on finitely generated group has a escape rate.

Next, Sawyer and Steger \cite{SS87} established the central limit theorem (CLT) for $l(Y_n)$. They showed that there exists some $\sigma^2 \geq 0$ such that
\begin{equation*}
	\lim_{n\to \infty} \frac{l(Y_n)-cn}{\sqrt{n}}  = N(0, \sigma^2) \quad  \text{in law}.
\end{equation*}
Moreover, Ledrappier provided a more geometric proof in his article \cite{Led01} and Bjorklund extended Ledrappier's argument to Gromov-hyperbolic groups and established the CLT for Green metric (\cite{GH13}).

The SLLN yields the escape rate of the RW, while the CLT further characterizes the "normal fluctuation" range of this rate. To analyze probabilities of "rare events" related to the escape rate, the Large Deviation Principle (LDP) becomes indispensable.  
Firstly, Lalley gave the asymptotic form of the transition probability of $(Y_n)$ \cite{Lal91} and then obtained a precise result on $(l(Y_n))$ \cite{Lal93}. 
Later, Corso \cite{Cor21} established the existence of the LDP of $(\frac{1}{n}l(Y_n))_{n\geq 1}$ with a good convex rate function. And he proved that the rate function coincides with the Fenchel-Legendre transform of the limiting logarithmic moment generating function of $(\frac{1}{n}l(Y_n))_{n\geq 1}$. 
Moreover, in \cite{LMW24}, Lai, Ma and Wang provided an alternative representation of the rate function.

While the LDP characterizes the tail behavior of rare events related to the escape rate, in many problems, the interest in rare events actually lies in probabilities about the whole sample path. 
  Bougerol and Jeulin \cite{BJ99} showed that if $(Y_n)$ is a simple random walk, then the conditional distribution of $( \frac{1}{\sqrt{n}}l(Y_{[nt]}), 0\leq t \leq 1)$ given $Y_0=Y_n=e$ converges to the Brownian excursion as $n\to\infty$.
  Also, Chen and Miermont \cite{CM17} proved that the range of a long Brownian bridge in the hyperbolic space converges after suitable renormalisation to the Brownian continuum random tree. 
   And in \cite{LMW24}, Lai, Ma and Wang proved that given the value of $l(Y_n)$, the value of $l(Y_k)$ at time $k$ is likely to be close to $\frac{k}{n}l(Y_n)$. 
   In our work, we establish the LDP about the sample path, namely the sample path LDP.

For the nearest random walk $(Y_n)_{n\geq 0}$ on the regular tree $T_d$, define
\begin{equation*}
  Z_n(t):=\frac{1}{n} l(Y_{[nt]}), \ 0\leq t \leq 1,
\end{equation*}
where $[c]$ is the integer part of $c$ and let $\mu_n$ be the law of $Z_n(\cdot)$ in $L_{\infty}([0,1])$. 

Denote by 
\begin{equation*}
	\mathcal{J} = \left\{ (t_1, t_2,\ldots, t_j)\in (0,1]^j: j \in \mathbb{N}, t_1<t_2<\cdots <t_j  \right\}
\end{equation*}
the collection of all ordered finite subsets of $(0,1]$. For any $J =(t_1, t_2,\ldots, t_j) \in \mathcal{J}$ and any $f \in  C_0([0,1])$, let $p_J(f) := (f(t_1), f(t_2), \ldots, f(t_{j}))$ and $\mu_n^J:=\mu_n \circ p_J^{-1}$,  i.e., $\mu_n^J$ is the law of the random vector $(Z_n(t_1),Z_n(t_2),\ldots, Z_n(t_{j}))$. 
The following Theorem \ref{thm-ldp} is our main result and it establishes the sample path LDP for $(Y_n)$ and gives an implicit expression of the rate function $I$. 

\begin{thm} \label{thm-ldp}
  The measures $(\mu_n )$ in $L_{\infty}([0,1])$  satisfy the LDP with the good rate function 
  \begin{equation*}
        I(f)= \begin{cases}
            \sup_{J\in \mathcal{J}} I_J(p_J(f)), & f \in  C_0([0,1])\\
            \infty, &  \text{otherwise}
        \end{cases}
    \end{equation*}
  where $I_J$ is the rate function of measures $(\mu_n^J)$. 
  
 Furthermore, $I_J$ is the Fenchel-Legendre transform of $\Lambda^J$ defined by
 \begin{equation*}
  	\Lambda^J(\lambda) := \lim_{n\to \infty} \frac{1}{n} \log \int_{ (\mathbb{R}_{\geq 0})^{j}}  e^{n \langle \lambda, x \rangle } d \mu_n^{J}(x),\quad \forall \lambda \in \mathbb{R}^{j}.
  \end{equation*}
 where $\langle \cdot , \cdot \rangle$ is the scalar product in $\mathbb{R}^{j}$. That is 
    \begin{equation*}
    I_J(x) = \sup_{\lambda \in \mathbb{R}^{j} } \{ \langle \lambda,x\rangle-\Lambda^J(\lambda)\} ,  \quad  \forall x \in (\mathbb{R}_{\geq 0})^{j}.
  \end{equation*}  
\end{thm}

\subsection{Overview of the proofs}
Similar to Mogulskii's theorem \cite[Theorem 5.1.2]{DZ09} on sample path LDP for RWs in $\mathbb{R}^d$,  let $\tilde{\mu}_n$ be the law of $\tilde{Z}_n(\cdot)$ in $L_{\infty}([0,1])$, where
\begin{equation*}
\tilde{Z}_n(t) := Z_n(t) + \left(t - \frac{[nt]}{n} \right)\left( l(Y_{[nt]+1}) - l(Y_{[nt]}) \right)
\end{equation*}	
is the polygonal approximation of  $Z_n(\cdot)$. Thus $\tilde{Z}_n(\cdot)$ is absolutely continuous and differentiable almost everywhere, and $\tilde{Z}_n(\cdot)$ is equal to $Z_n(\cdot)$ at the quantiles.  Furthermore, since $\|Z_n-\tilde{Z}_n\| \leq \frac{1}{n}$, we can obtain the exponential equivalence of the probability measures $(\mu_n)$ and $(\tilde{\mu}_n)$.

The establishment of sample path LDP mainly depends on the results of LDP for projective limits.
We show that the finite-dimensional projection $p_J(Z_n(\cdot))$ satisfies the LDP with a convex rate function $I_J$, which is the Fenchel-Legendre transform of the limit log-moment generating function $\Lambda^{J}$ of $ \mu_n^{J}$ (the existence of $\Lambda^{J}$ has also been proved). 
We prove the LDP for finite-dimensional projections by combining the weak LDP with the exponential tightness. For the proof of the weak LDP, we find out the typical path subset of the RW through several key time points, and split and reorganize the trajectories based on this subset. This differs from the method of direct geodesic concatenation technique for RWs  adopted by Corso \cite{Cor21}, thus avoiding the backsliding problem in the splicing process at multiple time points.

Due to the exponential equivalence of the probability measures $(\mu_n)$ and $(\tilde{\mu}_n)$ , the finite-dimensional projection $p_J(\tilde{Z}_n(\cdot))$ also satisfies the LDP with a convex rate function. 
Then by the Dawson-G{\"a}rtner theorem and the theorem on extension and restriction about function spaces and topologies, we get the LDP for $\tilde{Z}_n(\cdot)$.  Therefore the desired LDP for $Z_n(\cdot)$ can be obtained once again by the exponential equivalence.


\subsection{Outline of the article}
The paper is organized as follows. In section 2, we recall some results about RWs on regular trees, LDP and the sample path LDP, as well as some methods and tools in large deviation theory. Then, we complete the proof of the main result in section 3. Finally, in section 4, we give an example on the simple random walk and list list some questions about our results and propose  potential improvements and extensions for future research.

\section{Preliminaries}
In this section, we review some notations and results about random walks on regular tress and the large deviations. The readers are referred to \cite{Woe00,Lal23}  for a general introduction to random walks, and to \cite{DZ09} for an introduction of large deviations.

\subsection{Random walks on regular trees}
Let $G$ be the free product of $d$ two-element groups $\mathbb{Z}_2$, $d\geq 3$. That is, $G= (\mathbb{Z}_2)^{\ast d}$ can be represented by the set of all finite reduced words (including the empty word) from the alphabet $S=\{a_1,\ldots,a_d\}$ (a word is reduced if no letter $a_1,\ldots,a_d$ is adjacent to itself). Group multiplication in $G$ is concatenation followed by reduction, and the group identity $e$ is the empty word. Thus 
    for each $g=a_{i_1}\cdots a_{i_m} \in G$, $a_{i_j} \neq a_{i_{j+1}}$ for $j=1,\ldots,m-1$ and its inverse $g^{-1} = a_{i_{m}} a_{i_{m-1}} \cdots a_{i_1}$, i.e., $a_i^{-1}=a_i$ for $1\leq i\leq d$.

We denote by $l(g)$ the word length of $g \in G$ determined by the alphabet $S$, in particular, $l(g)$ is length of its representative word. We always assume that  $l(e)=0$. Then $l$ is a length function, meaning that it satisfies:
\begin{itemize}
  \item $l(g) \geq 0$ for all $g \in G$, and  $l(g)=0$ if and only if  $g=e$,
  \item $l(g)=l(g^{-1})$ for all  $g \in G$,
  \item $l(g_1 g_2) \leq l(g_1) + l(g_2)$, and equality holds if and only if the last letter of $g_1$ is different from the first letter of $g_2$. 
\end{itemize}

The Cayley graph of $(G,S)$ is a $d$-regular tree $T_d$. That is, the group identity $e$ is the root, each element $g \in G$ is an vertex and there is an edge between vertices $g_1, g_2$ if and only if $g_1^{-1} g_2 \in S$. In this paper, we shall not distinguish between group elements in $G$ and vertices in $T_d$. Thus for $g \in T_d$,  $l(g)$ is the graph distance between $g$ and the root.

Let $\mu=\{ p_i \}_{1 \leq i \leq d}$ be a probability distribution on $S$ with  $\mu(a_i)=p_i >0$ for each $1 \leq i \leq d$. We call $(Y_n)_{n \geq 0}$ a random walk (RW) on $T_d$ starting from $e$ with step distribution $\mu$ if $(Y_n)_{n \geq 0}$ is a Markov chain taking values in $T_d$ with transition probabilities 
 \begin{equation}
  \pr(Y_{n+1}=ga_i|Y_n=g)= p_i  \quad n\geq 0,
\end{equation}
and $Y_0 = e$. It is easy to see that  $(Y_n)_{n \geq 0}$ is a nearest neighbor  random walk. For convenience, we often think of a random walk in another way, as the multiplication of independent and identically distributed random variables on $G$. Let $\{X_n\}_{n \geq 1}$ be a sequence of independent and identically distributed random variables taking values in $S$ with law $\mu$. Then the random walk $(Y_n)_{n \geq 0}$ on $T_d$ starting from $e$ is 
\begin{equation*}
Y_n := \Pi_{i=1}^n X_i
\end{equation*}
with $Y_0 = e$.

\subsection{Large deviation and sample path large deviation }
In this part we recall some notions of large deviations and sample path large deviation. 

Throughout this subsection, $X$ denotes a locally convex Hausdorff topological space, endowed with the Borel $\sigma$-algebra $\mathcal{B}$ and $I: X \rightarrow [0,\infty]$ is a lower semicontinuous mapping, in particular, for all $\alpha \in [0, \infty]$, the level set $\Psi_{I}(\alpha):=\{x: I(x) \leq \alpha \}$ is a closed subset of $X$.  Additionally, if all the level sets $\Psi_{I}(\alpha)$ are compact subsets of $X$, we say $I$ is good.
The effective domain of $I$ is the set $D_I :=\{ x: I(x) <\infty \}$. 
 Let $(\mu_n )_{n\geq 1}$ be a sequence of Borel probability measures on $X$.

\begin{defn}
	We say that the probability measures sequence $( \mu_{n} )$  satisfies the large deviation principle, abbreviated as LDP, with a rate function $I$ if, for all $\Gamma \in \mathcal{B}$,
	\begin{equation}\label{defn-LDP}
		-\inf \limits_{x \in \Gamma^{o}}  I(x) \leq \liminf\limits_{n \to \infty}\frac{1}{n} \log \mu_{n}(\Gamma) \leq \limsup\limits_{n \to \infty}\frac{1}{n} \log \mu_{n}(\Gamma) \leq -\inf \limits_{x \in \overline{\Gamma}} I(x).
	\end{equation}	
	where $\Gamma^{o}$ and $\overline{\Gamma}$ are the interior and closure of $\Gamma$, respectively.
\end{defn}

After defining the LDP, a natural question is whether the rate function is unique, and the answer is yes due to the result of \cite[Lemma 4.1.4]{DZ09}.

 The LDP helps us to consider the tail behavior of many rare events associated with various sorts of empirical methods. However, in many problems, we are more concerned with rare events that depend on a set of random variables, or more generally, on a random process. Interest often lies in the probability that a path of a random process hits a particular set. Thus this kind of event requires us to consider the trajectory of the random process as a whole. 

Let $(Y_n)$ be a random walk on $T_d$ starting from $e$ with step distribution $\mu$. Define 
 \begin{equation}
  Z_n(t):=\frac{1}{n} l(Y_{[nt]}), \ 0\leq t \leq 1,
\end{equation}
where $[c]$ is the integer part of $c$. It is easy to verify that $Z_n(\cdot) \in L_{\infty}([0,1])$, where $L_{\infty}([0,1])$ is endowed with the supremum norm $\| \cdot \| $ and the supremum norm topology. Let $\mu_n$ be the law of  $Z_n(\cdot)$ in $L_{\infty}([0,1])$.
\begin{defn}
The sample path large deviation for random walk $(Y_n)$, abbreviated as sample path LDP,  is the large deviation for probability measures sequence $(\mu_n)$ in $L_{\infty}([0,1])$.
\end{defn}

\subsection{Some tools for LDP and sample path LDP}
For convenience, we often use a equivalent form of \eqref{defn-LDP}:
\begin{itemize}
        \item[(a)](Upper Bound) For every  $\alpha < \infty$ and every $\Gamma \in \mathcal{B}$ with  $\overline{\Gamma} \subset  X \setminus \Psi_{I}(\alpha)$,
        \begin{equation}\label{defn-LDP-upperbound}
            	\limsup\limits_{n\to \infty} \frac{1}{n} \log \mu_{n}(\Gamma) \leq -\alpha.
             \end{equation}
        \item[(b)](Lower Bound) For every $x \in D_I$ and every $\Gamma \in \mathcal{B}$ with $x \in \Gamma^{\circ}$,
         \begin{equation}\label{defn-LDP-lowerbound}
         	\liminf\limits_{n\to \infty} \frac{1}{n} \log \mu_{n}(\Gamma) \geq -I(x).
         \end{equation}
\end{itemize}	

Often a natural approach to prove the LDP is to first prove the upper bound and the lower bound in weaker forms (weak LDP) and then use the exponentially tightness of the probability measures.

\begin{defn}
We say that $( \mu_{n} )$ satisfies the weak large deviation principle, abbreviated as  weak LDP, with a rate function $I$ if the upper bound \eqref{defn-LDP-upperbound} holds for every $\alpha < \infty$ and all compact subsets of $ X  \backslash\Psi_{I}(\alpha)$, and the lower bound \eqref{defn-LDP-lowerbound} holds for all measurable sets.
\end{defn}

\begin{defn}
	We say that $ (\mu_{n})$ on $X$ is \textbf{exponentially tight} if for every $0 \leq \alpha < \infty$, there exists a compact set $K_{\alpha} \subset X$ such that 
	\begin{equation*}
  	\limsup_{n\to \infty} \frac{1}{n} \log \mu_{n}( X \backslash K_{\alpha} ) < -\alpha.
  \end{equation*}
\end{defn}

Then we can strengthen a weak LDP into a full LDP by the exponentially tightness.
\begin{prop}  \cite[Lemma 2.5]{LS87}\cite[Lemma 1.2.18]{DZ09} \label{thm-Dembo-lem1.2.18}
	Let $(\mu_{n})$ be an exponentially tight family on $X$. If  $(\mu_{n})$ satisfies the weak LDP with the rate function $I$, then $(\mu_{n})$ satisfies the LDP with the rate function $I$ and $I$ is a good rate function.
\end{prop}

In order to verify the weak LDP, we often use the following proposition without knowing the rate function in advance.
\begin{prop} \cite[Theorem 4.1.11]{DZ09} \label{thm-Dembo-thm4.1.11}
 Let $( \mu_{n} )$ be a sequence of Borel probability measures on $X$. For every $x\in X$, define the function $I: X \rightarrow [0,\infty]$ by
 \begin{equation*}
	I(x) := \sup_{V open: \, x \in V } -\liminf_{n\to \infty } \frac{1}{n} \log \mu_{n}(V).
\end{equation*}
Then $I$ is lower semicontinuous. Moreover, if for every $x\in X$, 
\begin{equation*}
	I(x) = \sup_{V open: \, x \in V } -\limsup_{n\to \infty } \frac{1}{n} \log \mu_{n}(V),
\end{equation*}
then the sequence $( \mu_{n} )$ satisfies the weak LDP with rate function $I$.
\end{prop}

Now we have tools to prove the LDP. It is known that the rate function of a sequence of i.i.d. random variables in empirical measure is the Fenchel-Legendre transform of log-moment generating function. For the general case, there are similar results. 

Let $X^{*}$  be the topological dual of $X$ and define the log-moment generating function of the measure $\mu$ on $X$ as the function $\Lambda_{\mu}$: $X^{*} \to (-\infty, \infty]$ given by
\begin{equation*}
	\Lambda_{\mu}(\lambda) :=\log \int_{X} e^{\langle\lambda, x \rangle}  \mu(dx) ,  \ \lambda \in X^{*}
\end{equation*}
where $\langle \cdot, \cdot \rangle$ is the standard dual pairing between $X^{*}$ and $X$. For a sequence of probability measures $(\mu_n)$, let
\begin{equation*}
	\bar{\Lambda} (\lambda) := \limsup_{n \to \infty} \frac{1}{n} \Lambda_{\mu_{n}} (n \lambda ), \quad \lambda \in X^{*}
\end{equation*}
and use the notation $\Lambda (\lambda)$ if the limit exists.

For $f : X \to (-\infty, \infty]$, not identically infinite, define its Fenchel-Legendre transform $f^{*} : X^{*} \to (-\infty, \infty]$ as
\begin{equation*}
    f^{*}(\lambda) =\sup_{x \in X} \{ \langle \lambda, x\rangle -f(x)\}, \  \varphi \in X^{*}.
\end{equation*}
In particular, if  $g: X^{*} \rightarrow (-\infty,\infty]$ is defined on the dual space $X^{*} $, we view its Fenchel-Lengendre transform $g^{*}$ as a function on $X$ rather than on $X^{\ast \ast}$.

The following Proposition \ref{thm-Dembo-thm4.5.10} is a consequence of Varadhan's integral lemma \cite{VS66} and Fenchel-Moreau's duality theorem  \cite[Theorem 1.11]{BB11} and it provides an expression of the rate function.

\begin{prop} \cite[Theorem 4.5.10]{DZ09} \label{thm-Dembo-thm4.5.10}
Let $(\mu_n )$ be a sequence of probability measures  on $X$ and satisfy the following conditions:
\begin{itemize}
  \item[\rm{1.}] For each $\lambda \in X^{*}$, $\bar{\Lambda}(\lambda)<\infty$,
  \item[\rm{2.}] $(\mu_n )$ satisfies the LDP with a good, convex rate function $I$.
\end{itemize}
Then for each $\lambda \in X^{*}$,the limit $\Lambda(\lambda) := \lim_{n \to \infty}  \Lambda_{\mu_{n}}(n \lambda )/n  $ exists, is finite, and the rate function $I$ and $\Lambda$ are the Fenchel-Legendre transform of each other, namely,
\begin{align*}
  I(x) &= \Lambda^{\ast}(x) =\sup_{\lambda \in X^{\ast} } \{ \langle \lambda, x\rangle-\Lambda(\lambda)\} ,  \  \forall x \in X \\
  \Lambda(\lambda) &= I^{\ast} (\lambda)=\sup_{x\in X} \{\langle \lambda, x\rangle -I(x) \}, \ \forall \lambda \in X^{*}.
\end{align*}
\end{prop}

The next two lemmas deals with the restriction and extension of the LDP, the first of which deals with the value space and the second with the topology. Thus under appropriate conditions, the LDP can be proved in a simple situation and then effortlessly transferred to a more complex one.

\begin{lemm} \cite[Lemma 4.1.5]{DZ09} \label{thm-Dembo-lem4.1.5}
	Let $H$ be a measurable subset of $X$ and $( \mu_n )$  a sequence of probability measures on $X$ with $\mu_{n}(H)=1$ for all $n\geq 1$. Suppose that $H$ is equipped with the topology induced by $X$.
	\begin{itemize}
    \item[(1)] If $H$ is a closed subset of $X$ and $(\mu_{n})$ satisfies the LDP in $H$ with rate function $I$, then  $(\mu_{n})$ satisfies the LDP in $X$ and now the rate function is $I'$:
    \begin{align*}
        I'(x)=\begin{cases}
            I(x),& x \in H \\
            \infty, & x \in X \backslash H
        \end{cases}
    \end{align*}
    \item[(2)] If $(\mu_{n})$ satisfies the LDP in $X$ with rate function $I$ and $D_I \subset H$, then the same LDP holds in $H$.  In particular, if $H$ is a closed subset of $X$, then $D_I \subset H$ and hence the LDP holds in $H$.\end{itemize}	
\end{lemm} 

\begin{lemm} \cite[Corollary 4.2.6]{DZ09} \label{thm-Dembo-cor4.2.6}
Suppose $(\mu_{n})$ is an exponentially tight family of probability measures on  $(X,\tau_1)$ and the topology $\tau_2$ on $X$ is coarser than $\tau_1$ ($\tau_2 \subset\tau_1$). If $(\mu_{n})$ satisfies an LDP on $(X,\tau_2)$, then $(\mu_{n})$ satisfies the same LDP on $(X,\tau_1)$.
\end{lemm}

Then, we illustrate that when two probability measures are sufficiently "close" (exponentially equivalent); they satisfy the same LDP.
\begin{defn}
Let $(X, d)$ be a metric space. The probability measures $(\mu_{n} )$ and $( \tilde{\mu}_{n} )$ on $X$ are called \textbf{exponentially equivalent} if there exist probability spaces $\{ (\Omega, \mathcal{B}_{n}, P_{n}) \}$ and two families of $X$-valued random variables $({Z_{n}} ), \, ( \tilde{Z}_{n})$ with joint laws $( P_{n} )$ and marginals $( \mu_{n} )$ and $ (\tilde{\mu}_{n})$, respectively, such that the following condition is satisfied:

For each $\delta>0$, the set $\{\omega:(Z_{n},\tilde{Z}_{n}) \in \Gamma_{\delta} \} \in \mathcal{B}_{n}$ and 
       \begin{equation*}
	   	\limsup_{n \to \infty} \frac{1}{n} \log P_{n}(\Gamma_{\delta})=-\infty,
	   \end{equation*}
where $\Gamma_{\delta}:= \{ (x,\tilde{x}): d(x, \tilde{x} )>\delta \} \subset X \times X$.
\end{defn}

\begin{prop} \cite[Theorem 4.2.13]{DZ09} \label{thm-Dembo-thm4.2.13}
	If an LDP with a good rate function $I$ holds for the probability measures  $(\mu_{n})$, which are exponentially equivalent to $(\tilde{\mu}_{n})$, then the same LDP holds for $(\tilde{\mu}_{n})$.
\end{prop}

Finally, we introduce the large deviation for projective limits which is the key tool to prove existence of the sample path LDP.

Let $(\mathcal{J}, \leq )$ be a partially ordered right-filtering set, i.e., for every $K,J \in \mathcal{J}$, there exist a $L \in\mathcal{J}$ such that $K \leq L$, $J \leq L$.  A projective system  $(Y_J, P_{KJ})_{K \leq J \in \mathcal{J}}$ consists of Hausdorff topological space $\{Y_J \}_{J \in \mathcal{J}}$ and continuous maps $\{P_{KJ}: Y_J \to Y_K \}_{K \leq J}$ such that $P_{KJ} = P_{KL} \circ P_{LJ}$ for $K \leq L \leq J$ and $P_{JJ} = \mathrm{id}_{Y_J}$.  The projective limit of this system is defined as 
\begin{equation*}
  \varprojlim Y_J:= \left\{ (y_J) \in \prod_{J \in \mathcal{J}} Y_J \ \bigg| \ \forall L \leq J,\ P_{LJ}(Y_J) = Y_L \right\},
\end{equation*}
with the topology induced by the product topology. The canonical projection $P_J: \varprojlim Y_J \to Y_J$ are continuous and the collection $\{ P_J^{-1}(U) \mid U \subseteq Y_J \text{ is open} \}$ is a base for the topology on $\varprojlim Y_J$.
Finally, we give the large deviation for projection limits.

\begin{prop}[\rm{Dawson-G{\"a}rtner}] \cite[Lemma 4.7]{DG87}  \cite[Theorem4.6.1]{DZ09}\label{thm-Dembo-thm4.6.1}
Let $(\mu_{n})$ be a sequence of probability measures on  $\varprojlim Y_J$ such that for each $J \in \mathcal{J}$, Borel probability measures $\mu_{n}\circ P_J^{-1}$ on $Y_J$ satisfy the LDP with the good rate function $I_J$. Then $(\mu_{n})$ satisfies the LDP with the good rate function
\begin{equation*}
  I(x)=\sup_{J\in \mathcal{J}} \{I_J(P_J(x))\} ,\ \ x \in \varprojlim Y_J
\end{equation*}
\end{prop}

\section{Sample Path Large Deviation for Random Walks}
In this section, we first show that the finite-dimensional projection at multiple time points of the sample path of the RW satisfies the LDP with a convex rate function. Our approach builds on Corso’s one-dimensional argument \cite{Cor21}, while handling multi-time projections by introducing stopping times tailored to the tree structure.
 We then apply the projection system theories to show the LDP in the projective limit space. 
 Finally, we establish the existence of the sample path LDP and provide an expression for the rate function.

\subsection{Path Concatenation Method} \label{sec-pinjie}
Let $S^n =\{ \bm{g} = (s_1, \ldots, s_n): s_i \in S, 1 \leq i \leq n \}$ be the set of $n$-step paths (using edges to represent paths). For $\bm{g} = (s_1, \ldots, s_n) \in S^n$, define $g(i) = s_1 \cdots s_i \in G$ as the reduced word corresponding to the path segment $\bm{g}(i) = (s_1, \ldots, s_i)$ for $1\leq i\leq n$. More generally, for $1\leq i \leq j \leq n$, we write $\bm{g}(i, j) = (s_i, \ldots, s_j)$ to denote the path from step $i$ to step $j$. 
The concatenation $\bm{g} \oplus \bm{g}'$ is defined as the direct joining of paths $\bm{g} = (s_1, \ldots, s_n)$ and $\bm{g}' = (s_1', \ldots, s_{n'}')$ without reduction, in particular,
\begin{equation*}
\bm{g} \oplus \bm{g}' = (s_1, \ldots, s_n, s_1', \ldots, s_{n'}').
\end{equation*}

Let $J=(t_1, t_2, \ldots, t_j)$ be an ordered subset of $(0,1]$ satisfying $0<t_1<t_2<\cdots <t_j \leq 1$. For $x= (x_1, x_2\ldots, x_j)\in (\mathbb{R}_{\geq 0})^j$ and $\rho > 0$,  we define the neighborhood of $x$ with radius $\rho$ as
\begin{equation*}
B(x, \rho) := \left\{ y=(y_1, y_2\ldots, y_j) \in \mathbb{R}^j : \max_{1\leq i\leq j}  |y_i - x_i|  < \rho \right\}
\end{equation*}
and consider the $n$-step paths confined to the $B(x, \rho)$, in particular, 
\begin{equation*}
\bm{F}_{n, \rho}^{J}(x) := \left\{ \bm{g}= (s_1, \ldots, s_{n}) \in S^{n} : \frac{1}{n} \left( l(g([n t_1])), \cdots, l(g([n t_j])) \right) \in B(x, \rho) \right\}.
\end{equation*}
Furthermore, for any positive integer $k$, we let $k B(x, \rho)=\{ k y: y \in B(x, \rho) \}$.

Recall that the RW $Y_n = X_1 X_2 \cdots X_n$. Let $\nu_n$ be the joint law of $(X_1, X_2, \ldots, X_{n})$. So it is easy to see that 
\begin{equation*}
	\pr \left( \frac{1}{n} \left( l(Y_{[n t_1]}), l(Y_{[n t_2]}), \ldots , l(Y_{[n t_j]}) \right) \in B(x,\rho) )  \right) = \nu_n \left( \bm{F}_{n, \rho}^{J}(x) \right).
\end{equation*}

We will show in Lemma \ref{lem-pinjie} that for any $\rho' > \rho$ and sufficiently large $n$, there exists a path set $\bm{B} \subset \bm{F}_{n, \rho}^{J}(x)$ with a lower bound on the exponential decay rate of the probability such that any $k$ paths from $\bm{B}$ can be concatenated to a new path with approximate length $k n$ by a specific concatenation method and the time-point projections of this new path fall into $k n B(x, \rho')$. 
For convenience, we assume that $j = 2$, i.e., $J=(t_1, t_2)$ and $x= (x_1, x_2)$. The path concatenation method we have developed can easily be extended to the case $j\geq 3$.

Before we state the path concatenation method, we introduce some notations that will be used next. For $\bm{g} \in \bm{F}_{n, \rho}^{J}(x)$, define 
\begin{align*}
y_1(\bm{g}) = g([n t_1]), \, \ 
y_2(\bm{g}) = g([n t_2]), \, \ 
(y_1 \wedge y_2)(\bm{g}) = y_1(\bm{g}) \wedge y_2(\bm{g}),
\end{align*}
\begin{align*}
L_1(\bm{g}) = l(g([n t_1])), \,  \ 
L_2(\bm{g}) = l(g([n t_2])), \,  \ 
L_0(\bm{g}) = l(y_1(\bm{g}) \wedge y_2(\bm{g})),
\end{align*}
\begin{align*}
m_1(\bm{g}) &= \max \left\{ i \leq [n t_1] : g(i) = (y_1 \wedge y_2)(\bm{g}) \right\}, \\
m_2(\bm{g}) &= \min \left\{ i \geq [n t_1] : g(i) = (y_1 \wedge y_2)(\bm{g}) \right\}.
\end{align*}
where $u\wedge v$ represents the most recent common ancestor of $u$ and $v$ in $T_d$. 
It is easy to see that $y_1$ and $y_2$ are the positions of the path at times $t_1$ and $t_2$, respectively, and $y_1 \wedge y_2$ is their most recent common ancestor.  And the quantities $L_0, L_1, L_2$ are their word lengths respectively.
Due to the cycle-free property of the tree, the path must visit the vertices in the order
\begin{equation*}
e \to y_1 \wedge y_2 \to y_1 \to y_1 \wedge y_2 \to y_2,
\end{equation*}
and $m_1, m_2$ record the last time and the first time that the path returns to $y_1 \wedge y_2$ before and after $t_1$, respectively.

Hence we can divide the path into four segments $e \to y_1 \wedge y_2$, $y_1 \wedge y_2 \to y_1$, $y_1 \to y_1 \wedge y_2$ and $ y_1 \wedge y_2 \to y_2$  with reduced words $ y_1 \wedge y_2$, $y_1$ and $y_2$ of lengths $L_0$, $L_1$ and $L_2$ respectively. 
Therefore we can decompose the set $\bm{F}_{n, \rho}^{J}(x)$ and get a path subset with fixed lengths $L_0,L_1,L_2,m_1,m_2$. Moreover, we can further decompose the path subset and get a path set  with a lower  bound of probability such that the paths share the same fixed initial and terminal letters  for each of the  reduced words associating with the segments. Then we can easily concatenate the paths in this path set with less reduction. Now we state the concatenation lemma.

\begin{lemm} \label{lem-pinjie}
Let	$J=(t_1, t_2)$ with $0< t_1 < t_2 \leq 1$,  $x= (x_1, x_2) \in (\mathbb{R}_{\geq 0})^2$ and $\rho'>\rho > 0$. Suppose
\begin{equation} \label{inequ-n_q-1}
n \geq \frac{4(x_1 + x_2 + 4)}{\rho' - \rho}.
\end{equation}
Then there exists a $n$-step path set $\bm{B} \subset \bm{F}_{n, \rho}^{J}(x)$ such that 
\begin{equation*}
	\nu_n (\bm{B}) \geq  \frac{1}{(n+1)^5 d^6}  \nu_n \left(\bm{F}_{n, \rho}^{J} (x) \right),
\end{equation*}
and there exist $a, b, c \in S$ such that for any paths $\bm{g}_1, \ldots, \bm{g}_k \in \bm{B}$, $k \geq 2$, the concatenated path $\bm{h} = \bm{h}(\bm{g}_1, \ldots, \bm{g}_k)$ has step length $\tilde{n}_k := k n + 4(k - 1)$ and satisfies
\begin{equation*}
\frac{1}{\tilde{n}_k} \left( l(h([\tilde{n}_k t_1])), \, l(h([\tilde{n}_k t_2])) \right) \in B(x, \rho'),
\end{equation*}
where $\bm{h} = \bm{h}(\bm{g}_1, \ldots, \bm{g}_k)$ is defined by
\begin{align*}
\bm{h} =\;& \bm{g}_1(m_1) \oplus a \oplus \bm{g}_2(m_1) \oplus a \oplus \cdots \oplus a \oplus \bm{g}_k(m_1) \oplus \\
& \bm{g}_1(m_1+1, [n t_1]) \oplus b \oplus \bm{g}_2(m_1+1, [n t_1]) \oplus b \oplus \cdots \oplus b \oplus \bm{g}_k(m_1+1, [n t_1]) \oplus \\
& \bm{g}_k([n t_1]+1, m_2) \oplus b \oplus \bm{g}_{k-1}([n t_1]+1, m_2) \oplus b \oplus \cdots \oplus b \oplus \bm{g}_1([n t_1]+1, m_2) \oplus \\
& \bm{g}_1(m_2+1, [n t_2]) \oplus c \oplus \bm{g}_2(m_2+1, [n t_2]) \oplus c \oplus  \cdots \oplus c \oplus \bm{g}_k(m_2+1, [n t_2]) \oplus \\
& \bm{g}_1([n t_2]+1, n) \oplus \bm{g}_2([n t_2]+1, n) \oplus \cdots \oplus \bm{g}_k([n t_2]+1, n).
\end{align*}
\end{lemm}

\begin{figure*}[htpt]
\centering
\includegraphics[width=6in]{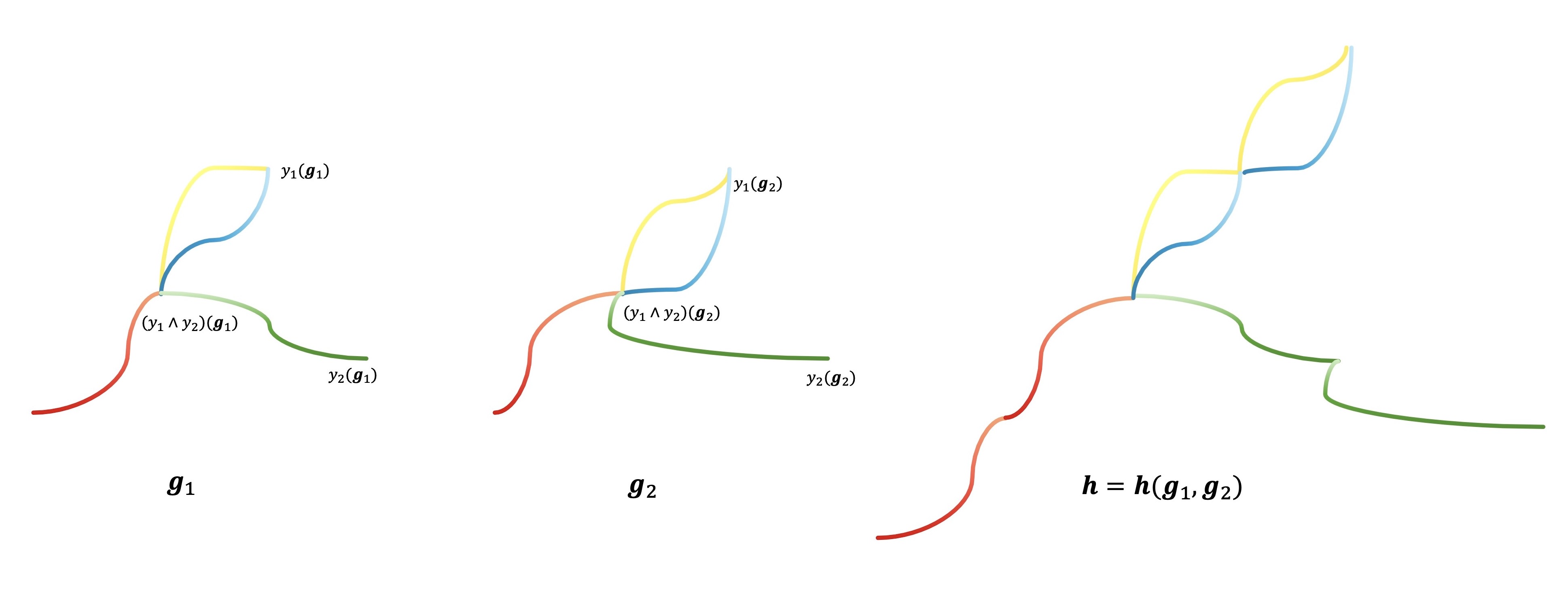}
\caption{Illustration for the concatenation of paths $\bm{g}_1$ and $\bm{g}_2$. }
\end{figure*}

\begin{proof}[Proof of Lemma \ref{lem-pinjie}]
 We first find the $n$-step path set $\bm{B}$ and then show that the paths in $\bm{B}$ satisfy the concatenation result.
 
Due to the values of $L_0(\bm{g}), L_1(\bm{g}), L_2(\bm{g}), m_1(\bm{g}), m_2(\bm{g})$,	the set $\bm{F}_{n, \rho}^{J}(x)$ can be decomposed into finite disjoint subsets, in particular,
\begin{equation*}
\bm{F}_{n, \rho}^{J} (x) = \bigcup_{L_0,L_1,L_2,m_1,m_2} \bm{A}(L_0,L_1,L_2,m_1,m_2),
\end{equation*} 
for $0 \leq  L_0, \, L_1, \, L_2, \, m_1, \, m_2 \leq n$, and 
\begin{eqnarray*}
\lefteqn{\bm{A}(L_0,L_1,L_2,m_1,m_2)} \\
&=& \left\{ \bm{g} \in \bm{F}_{n, \rho}(x) : L_0(\bm{g})=L_0,\, L_1(\bm{g})=L_1,\, L_2(\bm{g})=L_2, \, m_1(\bm{g})=m_1,\, m_2(\bm{g})=m_2 \right\}.
\end{eqnarray*}
Hence by the disjointness of the sets $\bm{A}(L_0,L_1,L_2,m_1,m_2)$, there exist integers $L_0,L_1,L_2,m_1,m_2$ such that 
\begin{equation*}
	\nu_{n} \left(\bm{A}(L_0,L_1,L_2,m_1,m_2) \right) \geq \frac{1}{(n+1)^5} \nu_n \left(\bm{F}_{n, \rho}^{J} (x) \right).
\end{equation*}
Similarly, for the fixed path set  $\bm{A}(L_0,L_1,L_2,m_1,m_2)$, we can further classify paths according to the first and last letters of the reduced words $y_1 \wedge y_2$, $(y_1 \wedge y_2)^{-1} y_1$ and $(y_1 \wedge y_2)^{-1} y_2$. Therefore there exists a subset 
\begin{equation*}
	\bm{B} \subset \bm{A}(L_0, L_1, L_2, m_1, m_2)
\end{equation*}
such that all paths in $\bm{B}$ share the same fixed initial and terminal letters (in $S$) for each of the three reduced segments above. Hence 
\begin{equation*}
\nu_n (\bm{B} ) \geq \frac{1}{d^{6}}  \nu_n \left(\bm{A}(L_0, L_1, L_2, m_1, m_2) \right) \geq \frac{1}{(n+1)^5 d^{6}} \nu_n \left(\bm{F}_{n, \rho}^{J} (x) \right)
\end{equation*}
and there exist  letters $a, b, c \in S$ such that for each path $\bm{g} \in\bm{B}$,
\begin{itemize}
  \item $a$ differs from both the first and last letters of $(y_1 \wedge y_2)(\bm{g})$,
  \item $b$ differs from both the first and last letters of $((y_1 \wedge y_2)(\bm{g}))^{-1} y_1(\bm{g})$,
  \item $c$ differs from both the first and last letters of $((y_1 \wedge y_2)(\bm{g}))^{-1} y_2(\bm{g})$.
\end{itemize}
When $L_0 =L_1$, we have that $y_1 \wedge y_2= y_1$ which implies  $m_1=[n t_1]=m_2$. Hence the corresponding path segments $\bm{g}(m_1+1, [n t_1])$ and  $\bm{g}([nt_1]+1, m_2 )$ are empty.  Therefore we view these as the empty path and pick up $b \in S$ arbitrarily.
When $L_0 =L_2$, we have that $y_1 \wedge y_2= y_2$. If $m_2=[nt_2]$, we view $\bm{g}(m_2+1, [nt_2])$ as the empty path and take $c$ which is not equal to the last letter of $y_1 \wedge y_2$ and the first of $(y_1 \wedge y_2)^{-1} y_1$. Otherwise, we have that $m_2<[nt_2]$ and the path $\bm{g}(m_2+1, [nt_2])$ formulate a loop. Therefore we fix the first and last letters of the path $\bm{g}(m_2+1, [nt_2])$ as the last letter of $y_1 \wedge y_2$ and also take $c$ which is not equal to the last letter of $y_1 \wedge y_2$ and the first of $(y_1 \wedge y_2)^{-1} y_1$.

Now we have found the path set $\bm{B}$ and then we show that the paths in $\bm{B}$ satisfy the concatenation result. For $\bm{g}_1, \ldots, \bm{g}_k \in \bm{B}$, it is easy to see the path $\bm{h} = \bm{h}(\bm{g}_1, \ldots, \bm{g}_k)$ has step length  $\tilde{n}_k = k n + 4(k - 1)$ and
\begin{align*}
h\left(k [n t_1] + 2(k - 1)\right) 
&=  ( y_1 \wedge y_2)( \bm{g_1})\cdot a \cdot  ( y_1 \wedge y_2)( \bm{g_2}) \cdot a \cdots a \cdot  ( y_1 \wedge y_2)( \bm{g_k}) \cdot \\
&\quad (y_1 \wedge y_2)^{-1} y_1(\bm{g}_1) \cdot b \cdot (y_1 \wedge y_2)^{-1} y_1(\bm{g}_2) \cdot b\cdots b \cdot (y_1 \wedge y_2)^{-1} y_1(\bm{g}_k).
\end{align*}
Hence the word length 
\begin{equation*}
	l\left(h\left(k [n_q t_1] + 2(k - 1)\right)\right) =2(k-1) + \sum_{i = 1}^{k} l(y_1(\bm{g}_i)) = 2(k - 1)+k L_1.
\end{equation*}
Since $g_i \in \bm{A}(L_0, L_1, L_2, m_1, m_2) \subset \bm{F}_{n, \rho}^{J} (x) $ for all $1\leq  i \leq k$, we have $|L_1 - n x_1| < n \rho$. Hence
\begin{eqnarray*}
\left| l(h([\tilde{n}_k t_1])) - \tilde{n}_k x_1 \right|  
& \leq & \left| l\left(h( [\tilde{n}_k t_1])\right) -l\left( h(k[n t_1]+ 2(k-1)) \right) \right|  \\
&& + \left|l\left( h(k[n t_1]+ 2(k-1)) \right) - kn x_1 \right| + \left|kn x_1 - \tilde{n}_k x_1 \right| 	 \\
&< & \left| [\tilde{n}_k t_1] - k[n t_1] - 2(k - 1) \right| + k n \rho + 2(k - 1) + 4(k - 1) x_1 \\
&\leq & \tilde{n}_k t_1 - k n t_1 + k + 1 + 2(k - 1) + k n \rho + 2(k - 1) + 4(k - 1) x_1 \\
&\leq & k n \rho + 4(k - 1)(x_1 + 3).
\end{eqnarray*}

Similarly, for $\tilde{n}_k t_2$, we have that 
\begin{equation*}
\left| l(h([\tilde{n}_k t_2])) - \tilde{n}_k x_2 \right| \leq   k n \rho + 4(k - 1)(x_2 + 4)
\end{equation*}
Therefore, by assumption \eqref{inequ-n_q-1}, 
\begin{equation*}
\left| l(h([\tilde{n}_k t_1])) - \tilde{n}_k x_1 \right| \leq  \tilde{n}_k \rho'   \quad \text{ and } \quad 
\left| l(h([\tilde{n}_k t_2])) - \tilde{n}_k x_2 \right| \leq  \tilde{n}_k \rho', 
\end{equation*}
which implies 
\begin{equation*}
\frac{1}{\tilde{n}_k} \left( l(h([\tilde{n}_k t_1])), \, l(h([\tilde{n}_k t_2])) \right) \in B(x, \rho').
\end{equation*}
\end{proof}

\subsection{Large Deviations for Finite-Dimensional Projections} \label{sec-finite}
Let $\mathcal{J}$ be the collection of all ordered finite subsets of $(0,1]$, formally,
\begin{equation*}
\mathcal{J} = \left\{ (t_1, t_2,\ldots, t_j)\in (0,1]^j: j \in \mathbb{N}, t_1<t_2<\cdots <t_j  \right\},
\end{equation*}
and $p_J$ the finite-dimensional projection for $J=\{ 0<t_1<t_2<\cdots<t_{j}\leq1 \} \in \mathcal{J}$, i.e., for any function $f:[0,1] \to \mathbb{R}$ , $p_J(f) = (f(t_1), f(t_2), \ldots, f(t_{j})) \in \mathbb{R}^j$.

Recall that 
\begin{equation*}
  Z_n(t)=\frac{1}{n} l(Y_{[nt]}), \ 0\leq t \leq 1,
\end{equation*}
and $\mu_n$ is the law of $Z_n(\cdot)$ in $L_{\infty}([0,1])$. Define $\mu_n^J :=\mu_n \circ p_J^{-1}$ and it is easy to see that $\mu_n^{J}$ is the law of the random vector 
 \begin{equation*}
 	(Z_n(t_1),Z_n(t_2),\ldots, Z_n(t_{j})).
 \end{equation*}
Note that $\nu_n$ is the joint law of $(X_1, X_2, \ldots , X_n)$. Hence
\begin{equation}\label{equ-mu-nu}
\mu_n^{J} (B(x,\rho)) =	\pr \left( \frac{1}{n} \left( l(Y_{[n t_1]}), l(Y_{[n t_2]}), \ldots , l(Y_{[n t_j]}) \right) \in B(x,\rho) )  \right) = \nu_n \left( \bm{F}_{n, \rho}^{J}(x) \right).
\end{equation}

The following proposition establishes the LDP for the multi-time projections of the sample path function $Z_n(\cdot)$, specifically for the probability measures sequence $(\mu_n^J)$.

\begin{prop} \label{prop-mu_n^J-ldp}
The sequence $( \mu_n^J )$ satisfies the LDP in $\mathbb{R}^j$ with the good rate function
\begin{equation*}
I_J(x) = \sup_{V open: x\in V} \left( -\limsup_{n \to \infty} \frac{1}{n} \log \mu_n^J(V) \right)=\sup_{V open: x\in V} \left( -\liminf_{n \to \infty} \frac{1}{n} \log \mu_n^J(V) \right).
\end{equation*}
\end{prop}

Before we prove Proposition \ref{prop-mu_n^J-ldp}, we give a lemma that enables us to derive a lower bound on the asymptotic exponential decay rate of the probabilities $\mu_n^{J} (B(x, \delta))$, provided a uniform lower bound over a non-lacunary sequence of times $(n_m)$.

\begin{lemm} \label{lem-mu_n^J-lowerbound}
Suppose there exist $ \rho > 0$, $\gamma \in \mathbb{R}$, and a strictly increasing sequence $(n_m)$ with
\begin{equation*}
\lim_{m \to \infty} \frac{n_{m+1}}{n_m} = 1,
\end{equation*}
such that
\begin{equation*}
\mu_{n_m}^J(B(x, \rho)) \geq e^{n_m \gamma} \quad \text{for all } m \geq 1.
\end{equation*}
Then for any $\delta > \rho $,
\begin{equation*}
\liminf_{n \to \infty} \frac{1}{n} \log \mu_n^J(B (x, \delta)) \geq \gamma.
\end{equation*}
\end{lemm}

\begin{proof}[Proof of Lemma \ref{lem-mu_n^J-lowerbound}]
Since $(n_m)$ is strictly increasing and $\lim_{m \to \infty} n_{m+1}/n_m = 1$,
there exists a unique $m = m(N)$ such that $n_m \leq N < n_{m+1}$ for each $N \geq n_1$. And there also exists $m_0 \in \mathbb{N}$ such that for all $k > m_0$,
\begin{equation*}
n_{k+1} - n_k < \frac{\delta - \rho}{2 + \max_{1\leq i\leq j} x_i + \max_{1\leq i\leq j} t_i} \cdot n_k.
\end{equation*}

Assume $m = m(N) > m_0$ and
\begin{equation*}
(Z_{n_m}(t_1), \ldots, Z_{n_m}(t_j)) \in B (x, \rho).
\end{equation*}
Then, for each $1\leq i\leq j$,
\begin{align*}
|l(Y_{[N t_i]}) - N x_i| 
&\leq |l(Y_{[N t_i]}) - l(Y_{[n_m t_i]})| + |l(Y_{[n_m t_i]}) - n_m x_i| + |n_m x_i - N x_i| \\
&< (n_{m+1} - n_m)(2 + x_i + t_i) + n_m \rho \leq \delta N.
\end{align*}
Thus,
\begin{equation*}
\mu_N^J(B (x, \delta )) \geq \mu_{n_m}^J(B(x, \rho)) \quad \Rightarrow \quad
\frac{1}{N} \log \mu_N^J(B (x, \delta )) \geq \frac{n_m}{N} \gamma.
\end{equation*}
Therefore, taking $N \to \infty$ which satisfies $m = m(N) > m_0$, we get the result
\begin{equation*}
\liminf_{n \to \infty} \frac{1}{n} \log \mu_n^J(B(x, \delta )) \geq \gamma.
\end{equation*}
\end{proof}

Now we can prove the existence of the LDP for $( \mu_n^J )$. 

\begin{proof}[Proof of Proposition \ref{prop-mu_n^J-ldp}]
By Proposition \ref{thm-Dembo-lem1.2.18}, it suffices to verify that the sequence $(\mu_n^J )$ is exponentially tight and satisfies the weak LDP.

\textbf{Exponential tightness.}  
Fix $M > 1$. Since $Z_n(t_i) \leq t_i \leq 1$ for all  $1\leq i \leq j$, we have that
$\mu_n^J(([0, M]^j)^C) = 0$. Hence,
\begin{equation*}
\lim_{M \to \infty} \frac{1}{n} \log \mu_n^J(([0, M]^j)^C) = -\infty.
\end{equation*}
Thus $(\mu_n^J )$ is exponentially tight due to the compactness of  $[0, M]^j \subset \mathbb{R}^j$.

\textbf{Weak LDP.}  
For every $x \in (\mathbb{R}_{\geq 0})^j$, define 
\begin{equation*}
I(x) := \sup_{V open: x\in V} -\liminf_{n \to \infty} \frac{1}{n} \log \mu_n^J(V).
\end{equation*}
Then by  Proposition \ref{thm-Dembo-thm4.1.11}, we only have to show
\begin{equation} \label{equ-I(x)-sup}
I(x) = \sup_{V open: x\in V}  -\limsup_{n \to \infty} \frac{1}{n} \log \mu_n^J(V) .
\end{equation}

For convenience, we assume $j=2$, i.e., $J=(t_1, t_2)$. And the case $j\geq 3$ is similar. Suppose \eqref{equ-I(x)-sup} does not hold. Then there is $x =(x_1,x_2)  \in (\mathbb{R}_{\geq 0})^2$ such that 
\begin{equation*}
	I(x) >\sup_{V open: x\in V}  -\limsup_{n \to \infty} \frac{1}{n} \log \mu_n^J(V). 
\end{equation*}
Hence there exist $\delta, \eta > 0$ such that
\begin{equation*}
\liminf_{n \to \infty} \frac{1}{n} \log \mu_n^J(B(x, \delta)) 
< \inf_{\rho > 0} \limsup_{n \to \infty} \frac{1}{n} \log \mu_n^J(B(x, \rho)) - \eta.
\end{equation*}
Fix $0 <\rho < \delta$. Then there exists a strictly increasing subsequence $(n_m)_{ m\geq 1}$  such that 
\begin{equation*}
	\alpha < \beta_{n_m} - \eta,
\end{equation*}
where 
\begin{equation*}
\alpha = \liminf_{n \to \infty} \frac{1}{n} \log \mu_n^J (B(x, \delta)) \, \text{ and } \, \beta_{n_m} = \frac{1}{n_m} \log \mu_{n_m}^J (B(x, \rho)).	
\end{equation*}

By Lemma \ref{lem-pinjie}, we know that for $n_m$ large enough, there exits a $n_m$-step path set $\bm{B} \subset \bm{F}_{n_m, \rho}^{J}(x)$ such that 
\begin{equation*}
	\nu_{n_m} (\bm{B}) \geq  \frac{1}{(n_m+1)^5 d^6}  \nu_{n_m} \left(\bm{F}_{n_m, \rho}^{J} (x) \right) = \frac{1}{(n_m + 1)^5 d^6} e^{n_m \beta_{n_m}},
\end{equation*}
where the right equality is due to the equality \eqref{equ-mu-nu}.
And for paths $\bm{g}_1, \ldots, \bm{g}_k \in \bm{B}$, the concatenation $\bm{h} = \bm{h}(\bm{g}_1, \ldots, \bm{g}_k)$ has step length $\tilde{n}_{m_k} = k n_m + 4(k - 1)$  and satisfies
\begin{equation*}
\frac{1}{\tilde{n}_{m_k}} \left( l(h([ \tilde{n}_{m_k} t_1])), \, l(h([\tilde{n}_{m_k} t_2])) \right) \in B(x, \rho')
\end{equation*}
for $\rho' >\rho$. Hence we have 
\begin{eqnarray*}
	\mu_{\tilde{n}_{m_k}}^J (B(x, \rho')) \geq \nu_{n_m} (\bm{B})^k p^{4(k - 1)} 
	\geq \left( \frac{1}{(n_m + 1)^5 d^6} e^{n_m \beta_{n_m}} \right)^k p^{4(k - 1)},
\end{eqnarray*}
where $p=\min_{1\leq i \leq d} p_i >0 $.  For $n_m$ large enough, we have that $n_{m}\eta \geq 2(\beta_{n_m}-\eta)+ 5 \log (n_{m}+1) + 6\log d - 4 \log p$.
Therefore, 
\begin{equation*}
	\mu_{\tilde{n}_{m_k}}^J(B(x, \rho')) \geq e^{\tilde{n}_{m_k} (\beta_{n_m} - \eta)}.
\end{equation*}
Then taking $\rho'\in (\rho, \delta)$ and applying Lemma~\ref{lem-mu_n^J-lowerbound}, we conclude that
\begin{equation*}
\alpha =\liminf_{n \to \infty} \frac{1}{n} \log \mu_n^J (B(x, \delta)) \geq \beta_{n_m} - \eta,
\end{equation*}
which is a contradiction. Hence equation \eqref{equ-I(x)-sup} holds.

Therefore, $( \mu_n^J )$ satisfies the LDP in $\mathbb{R}^j$ with the good rate function
\begin{equation*}
I_J(x) = \sup_{V open: x\in V} \left( -\limsup_{n \to \infty} \frac{1}{n} \log \mu_n^J(V) \right)=\sup_{V open: x\in V} \left( -\liminf_{n \to \infty} \frac{1}{n} \log \mu_n^J(V) \right).
\end{equation*}
\end{proof}

For the rate function $I_J$, we can express it as the Fenchel-Legendre transform of the log-moment generating function. 
Recall that the log-moment generating function of the probability measure $\mu_n^J$ is defined by
\begin{equation*}
	 \Lambda^J_n(\lambda) :=\log  \int_{ (\mathbb{R}_{\geq 0})^j} e^{\langle \lambda,x \rangle}  \mu_n^J (d x) \quad \lambda \in \mathbb{R}^j,
\end{equation*}
where  $\langle \cdot , \cdot \rangle$ is the scalar product in $\mathbb{R}^{j}$.

\begin{lemm} \label{lem-mu_n^J-ratefun}
For every $\lambda \in \mathbb{R}^j$, the limit $\Lambda^J(\lambda) =\lim_{n \to \infty} \frac{1}{n} \Lambda^J_n( n\lambda)$ exists and is finite. Furthermore, $I_J$ is the Fenchel-Legendre transform of $\Lambda^J$, namely,
\begin{equation*}
I_J(x) = \sup_{\lambda \in \mathbb{R}^j} \left\{ \langle \lambda, x \rangle - \Lambda^J(\lambda) \right\}, \quad \forall x \in (\mathbb{R}_{\geq 0})^j.
\end{equation*}
\end{lemm}

\begin{proof}[Proof of Lemma \ref{lem-mu_n^J-ratefun}]
By Proposition \ref{prop-mu_n^J-ldp}, the sequence $(\mu_n^J)$ satisfies the LDP with  the good rate function $I_J$. Hence due to Proposition \ref{thm-Dembo-thm4.5.10}, it suffices to show that $I_J$ is convex and the upper limit log-moment generating function 
\begin{equation*}
	\bar{\Lambda}^J(\lambda) :=\limsup_{n \to \infty} \frac{1}{n}\Lambda^J_n( n\lambda)
\end{equation*}
is finite everywhere. 

For each $\lambda = (\lambda_1, \ldots, \lambda_j)$, we have
\begin{eqnarray*}
	\bar{\Lambda}^J(\lambda) &=& \limsup_{n \to \infty} \frac{1}{n}\log  \int_{ (\mathbb{R}_{\geq 0})^j} e^{ n\langle \lambda,x \rangle}  \mu_n^J (d x) 
	= \limsup_{n \to \infty} \frac{1}{n} \log \mathbb{E} \left[ e^{n\sum_{i=1}^j \lambda_i Z_n(t_i)} \right] \\
	&=& \limsup_{n \to \infty} \frac{1}{n} \log \mathbb{E} \left[ e^{\sum_{i=1}^j \lambda_i l(Y_{[n t_i]})} \right] 
	\leq  \limsup_{n \to \infty} \frac{1}{n} \sum_{i=1}^j |\lambda_i| [n t_i]
	\leq  \sum_{i=1}^j |\lambda_i| t_i < \infty.
\end{eqnarray*}
Therefore, it remains to prove that $I_J$ is convex. We first verify the midpoint convexity of $I_J$: for any $x, y \in (\mathbb{R}_{\geq 0})^j$,
\begin{equation*}
I_J\left( \frac{1}{2} x + \frac{1}{2} y \right) \leq \frac{1}{2} I_J(x) + \frac{1}{2} I_J(y).
\end{equation*}

Suppose for contradiction that there exist $x, y \in (\mathbb{R}_{\geq 0})^j$ such that
\begin{equation} \label{inequ-midconvex}
I_J\left( \frac{1}{2} x + \frac{1}{2} y \right) > \frac{1}{2} I_J(x) + \frac{1}{2} I_J(x).
\end{equation}
Recall that by Proposition \ref{prop-mu_n^J-ldp},
\begin{equation*}
I_J(x) = \sup_{V open: x\in V} \left( -\limsup_{n \to \infty} \frac{1}{n} \log \mu_n^J(V) \right) = \sup_{V open: x\in V} \left( -\liminf_{n \to \infty} \frac{1}{n} \log \mu_n^J(V) \right).
\end{equation*}
Then \eqref{inequ-midconvex} implies that there exist $\delta, \eta > 0$ such that for any $\rho_1, \rho_2 > 0$,
\begin{align} \label{inequ-midconvex1}
\limsup_{n \to \infty} & \frac{1}{n} \log \mu_n^J\left( B\left( \frac{1}{2} x + \frac{1}{2} y, \delta \right) \right) \\
&< \frac{1}{2} \liminf_{n \to \infty} \frac{1}{n} \log \mu_n^J(B(x, \rho_1)) 
 + \frac{1}{2} \liminf_{n \to \infty} \frac{1}{n} \log \mu_n^J(B(y, \rho_2)) - \eta.
\end{align}

For simplicity, we assume $j = 2$ and choose $\rho_1 = \rho_2 = \rho < \delta$. 
We claim that there exists $n_0$ such that for all $n > n_0$ and $\tilde{n} = 2n + 4$,
\begin{align*}
\frac{1}{\tilde{n}} \log \mu_{\tilde{n}}^J \left( B\left( \frac{1}{2} x + \frac{1}{2} y, \delta \right) \right)
\geq \frac{1}{2} \left( \frac{1}{n} \log \mu_n^J(B(x, \rho)) + \frac{1}{n} \log \mu_n^J(B(y, \rho)) \right) - \eta.
\end{align*}
This contradicts \eqref{inequ-midconvex1}. Hence, we get the midpoint convexity of $I_J$.

We now justify the claim. Following the same strategy as in Proposition \ref{prop-mu_n^J-ldp}, define the $n$-step paths sets $\bm{F}_{n, \rho}^{J}(x), \, \bm{F}_{n, \rho}^{J}(y)$, and let
\begin{equation*}
\beta_x = \frac{1}{n} \log \mu_n^J(B(x, \rho)), \quad 
\beta_y = \frac{1}{n} \log \mu_n^J(B(y, \rho)).
\end{equation*}
Hence, by Lemma \ref{lem-pinjie}, there exist subsets $\bm{B}(x) \subset \bm{F}_{n, \rho}^{J}(x)$ and $\bm{B}(y) \subset \bm{F}_{n, \rho}^{J}(y)$ such that 
\begin{equation*}
v_n(\bm{B}(x)) \geq \frac{1}{(n + 1)^5 d^6} e^{n \beta_x}, \quad
v_n(\bm{B}(y)) \geq \frac{1}{(n + 1)^5 d^6} e^{n \beta_y},
\end{equation*}
and there exist $a, b, c \in S$ ensuring the cancellation-free concatenation.

For $\bm{g}_x \in \bm{B}(x)$ and $\bm{g}_y \in \bm{B}(y)$, define the concatenated path $\bm{h} = \bm{h}(\bm{g}_x, \bm{g}_y)$ by
\begin{align*}
\bm{h}(\bm{g}_x, \bm{g}_y)
=& \bm{g}_x(m_1^x) \oplus a \oplus \bm{g}_y(m_1^y) \oplus
\bm{g}_y(m_1^y+1, [n t_1]) \oplus b \oplus \bm{g}_x(m_1^x+1, [n t_1]) \oplus \\
& \bm{g}_x([n t_1]+1, m_2^x) \oplus b \oplus \bm{g}_y([n t_1]+1, m_2^y) \oplus
\bm{g}_y(m_2^y+1, [n t_2]) \oplus c \oplus \bm{g}_x(m_2^x+1, [n t_2]) \oplus \\
& \bm{g}_x([n t_2]+1, n) \oplus \bm{g}_y([n t_2]+1, n).
\end{align*}
Then the path $\bm{h}$ has step length  $\tilde{n} = 2n + 4$ and satisfies
\begin{equation*}
\frac{1}{\tilde{n}} \left( l(h([\tilde{n} t_1]), \, l(h([\tilde{n} t_2]) \right) \in B \left(\frac{x + y}{2}, \delta \right).
\end{equation*}
Hence, the paths set 
\begin{equation*}
\bm{H} := \left\{ \bm{\tilde{h}}(\bm{g}_x, \bm{g}_y) : \bm{g}_x \in \bm{B}(x), \,\bm{g}_y \in \bm{B}(y) \right\} \subset \bm{\tilde{F}}_{\tilde{n}, \delta}^{J}\left(\frac{x + y}{2}\right),
\end{equation*}
Therefore for $n$ large enough,
\begin{eqnarray*}
	\mu_{\tilde{n}}^J \left( B \left( \frac{x + y}{2}, \delta \right) \right)
	&=& \nu_{\tilde{n}} \left(  \bm{\tilde{F}}_{\tilde{n}, \delta}^{J} \left(\frac{x + y}{2}\right)\right) \geq  \nu_{\tilde{n}}(\bm{H}) 
	\geq  v_n(\bm{B}(x)) \cdot v_n(\bm{B}(y)) \cdot p^{4} \\
	&\geq &  \frac{p^4}{(n + 1)^{10} d^{12}} e^{n(\beta_x + \beta_y)} \geq e^{\tilde{n} \left( \frac{\beta_x + \beta_y}{2} - \eta \right)},
\end{eqnarray*}
where $p = \min_{1\leq i \leq d} p_i>0$. This proves the claim. And we get the midpoint convexity.

Therefore, since $I_J$ is lower semicontinuous and midpoint convex, it is convex by iteration. Hence, the log-moment generating function 
\begin{equation*}
	\Lambda^J(\lambda) =\lim_{n \to \infty} \frac{1}{n} \Lambda^J_n( n\lambda) =\lim_{n \to \infty} \frac{1}{n}\log  \int_{ (\mathbb{R}_{\geq 0})^j} e^{ n\langle \lambda,x \rangle}  \mu_n^J (dx) = \lim_{n \to \infty} \frac{1}{n} \log \mathbb{E} \left[ e^{\sum_{i=1}^j \lambda_i l(Y_{[n t_i]})} \right]<\infty 
\end{equation*}
for every $\lambda \in \mathbb{R}^j$ and the rate function has the expression
 \begin{equation*}
 	I_J(x) = \sup_{\lambda \in \mathbb{R}^j} \left\{ \langle \lambda, x \rangle - \Lambda^J(\lambda) \right\}, \quad \forall x \in (\mathbb{R}_{\geq 0})^j.
 \end{equation*}
\end{proof}

\subsection{Large Deviations in the Projective Limit Space} \label{sec-touying}
 Recall that $Z_n(\cdot)$ is the sample path function of the random walk. Define 
\begin{equation*}
\tilde{Z}_n(t) := Z_n(t) + \left(t - \frac{[nt]}{n} \right)\left( l(Y_{[nt]+1}) - l(Y_{[nt]}) \right), \ 0\leq t \leq 1.
\end{equation*}
and let $\tilde{\mu}_n$ be the law of $\tilde{Z}_n(\cdot)$ in $L_{\infty}([0,1])$. By construction, it is easy to see that $\tilde{Z}_n(\cdot)$ is absolutely continuous, and it connects the $n$-quantile points of $Z_n(\cdot)$ and they coincides at those points.

 We first show that $(\tilde{\mu}_n)$ and $(\mu_n)$ are exponentially equivalent. Therefore, the LDP for $(\mu_n)$ is the same as that for $(\tilde{\mu}_n)$.

\begin{lemm} \label{lem-expequ}
The probability measures $(\mu_n)$ and $(\tilde{\mu}_n)$ are exponentially equivalent in $L_\infty([0,1])$.
\end{lemm}

\begin{proof}[Proof of Lemma~\ref{lem-expequ}]
For any $\eta > 0$, it is obvious that the set $\{\omega: \| \tilde{Z}_n - Z_n \| > \eta\}$ is measurable. By the definition of $\tilde{Z}_n$, we have $|\tilde{Z}_n(t) - Z_n(t)| < \frac{1}{n}$. Hence, for $n > 1/\eta$,
\begin{equation*}
\mathbb{P}(\| \tilde{Z}_n - Z_n \| > \eta) = 0,
\end{equation*}
which implies
\begin{equation*}
\limsup_{n \to \infty} \frac{1}{n} \log \mathbb{P}(\| \tilde{Z}_n - Z_n \| > \eta) = -\infty.
\end{equation*}
Therefore, $(\tilde{\mu}_n)$ and $(\mu_n)$ are exponentially equivalent.
\end{proof}

Then we use the Dawson-G{\"a}rtner theorem to establish the LDP on the projective limit space. Recall that the collection of all ordered finite subsets of $(0,1]$ is
\begin{equation*}
\mathcal{J} = \left\{ (t_1, t_2,\ldots, t_j)\in (0,1]^j: j \in \mathbb{N}, t_1<t_2<\cdots <t_j  \right\}.
\end{equation*}
and we can define a partial order $\preccurlyeq$ on $\mathcal{J}$ by: $K = \{s_1 < \cdots < s_k\} \preccurlyeq J = \{t_1 < \cdots < t_j\}$ if for each $s_\ell$ there exists $t_{q(\ell)}$ such that $s_\ell = t_{q(\ell)}$.  Hence there exists a nature projective mapping 
\begin{equation*}
p_{KJ}: \mathbb{R}^j \to \mathbb{R}^k
\end{equation*}
for $K \preccurlyeq J$ and the projective limit  $\tilde{\mathcal{X}}$ of the system $(\mathcal{Y}_J = \mathbb{R}^j, p_{KJ} )$, i.e.,
\begin{equation*}
\tilde{\mathcal{X}} := \varprojlim \mathcal{Y}_J = \left\{ (y_K)_{K \in \mathcal{J}} \in \prod_{K \in \mathcal{J}} \mathcal{Y}_K \mid y_K = p_{KJ}(y_J) \text{ whenever } K \preccurlyeq J \right\}.
\end{equation*}

Let $\mathcal{X}$ be the space of real-valued functions on $[0,1]$ starting at $0$, endowed with the pointwise convergence topology, in particular,
\begin{equation*}
\mathcal{X} := \{ f: [0,1] \to \mathbb{R} \mid f(0) = 0 \}.
\end{equation*}
Then, we can identify the space $\tilde{\mathcal{X}}$ with $\mathcal{X}$. Indeed, for each $f \in \mathcal{X}$, the tuple $(p_J(f))_{J \in \mathcal{J}}$ lies in $\tilde{\mathcal{X}}$, where $p_J(f) = (f(t_1), f(t_2), \ldots, f(t_{j}))$. Conversely, any $(x_J)_{J \in \mathcal{J}} \in \tilde{\mathcal{X}}$ determines a function $f \in \mathcal{X}$ by $f(t) := x_{\{t\}}$ for $t > 0$ and $f(0) := 0$. 

It is easy to see that $\tilde{Z}_n(\cdot) \in \mathcal{X}$. Thus the probability measures $(\tilde{\mu}_n)$ can be naturally embedded into $\mathcal{X}$, and hence into $\tilde{\mathcal{X}}$. Therefore, we can use the Dawson-G{\"a}rtner theorem.

\begin{lemm} \label{lem-tildemu_n-ldp}
The probability measures sequence $(\tilde{\mu}_n)$ (defined on $\mathcal{X}$ by the natural embedding) satisfies the LDP with the rate function
\begin{equation*}
I_{\mathcal{X}}(f) = \sup_{J \in \mathcal{J}} I_J(p_J(f)), \quad f \in \mathcal{X}.
\end{equation*}
\end{lemm}

\begin{proof}[Proof of Lemma~\ref{lem-tildemu_n-ldp}]
Since the function space $\mathcal{X}$ is identified with the projective limit space $\tilde{\mathcal{X}}$, the pointwise convergence topology of $\mathcal{X}$ coincides with the projective topology on $\tilde{\mathcal{X}}$ and $p_J$ are the canonical projections for $\tilde{\mathcal{X}}$. Then by Proposition \ref{thm-Dembo-thm4.2.13}, Lemma \ref{lem-expequ}, and Proposition \ref{prop-mu_n^J-ldp}, the projective measures $\tilde{\mu}_n^J := \tilde{\mu}_n \circ p_J^{-1}$ satisfy the LDP in $\mathbb{R}^j$ with rate function $I_J$. Hence applying the Dawson–Gärtner theorem (Proposition \ref{thm-Dembo-thm4.6.1}), the sequence $(\tilde{\mu}_n)$ satisfies the LDP on $\tilde{\mathcal{X}}$ with the rate function
\begin{equation*}
I_{\tilde{\mathcal{X}}}(f) = \sup_{J \in \mathcal{J}} I_J(p_J(f)), \quad f \in \tilde{\mathcal{X}}.
\end{equation*}
Therefore, $(\tilde{\mu}_n)$ satisfies the LDP on $\mathcal{X}$ with the same rate function
\begin{equation*}
I_{\mathcal{X}}(f) = \sup_{J \in \mathcal{J}} I_J(p_J(f)), \quad f \in \mathcal{X}.
\end{equation*}
\end{proof}

\subsection{Proof of Theorem \ref{thm-ldp}}
Before we prove the Theorem \ref{thm-ldp}, we first show that the probability measures sequence $(\tilde{\mu}_n)$ is exponentially tight in $C_0([0,1])$ endowed with the topology of uniform convergence, where  $C_0([0,1])$ is the space of real-valued continuous functions on $[0,1]$ starting at $0$, in particular,  
\begin{equation*}
	C_0([0,1]):=\{f: [0,1] \to \mathbb{R} \mid f(0)=0 \text{ and $f$ continuous} \}.
\end{equation*}

\begin{lemm} \label{lem-tildemu_n-exptig}
Let $\tau_1$ be the topology of uniform convergence on $C_0([0,1])$. Then the sequence $(\tilde{\mu}_n )$ is exponentially tight in $(C_0([0,1]), \tau_1)$.
\end{lemm}

\begin{proof}[Proof of Lemma  \ref{lem-tildemu_n-exptig}]
For any $\alpha>0$, define 
\begin{equation*}
K_{\alpha} := \{f \in C_0[0,1]: |f(x)-f(y)|\leq (1+\alpha)|x-y|, \forall x, y \in [0,1] \}.	
\end{equation*}
Now we show that $K_{\alpha}$ is the set we wanted. 

Firstly, we show that $K_{\alpha}$ is a compact subset of $C_0[0,1]$. By Arzel\`a-Ascoli Theorem , we only have to show that $K_\alpha$ is equicontinuous, uniformly bounded and closed.

\textbf{Closed set.} Supposing $\{f_n\}\subset K_\alpha$ and uniformly converging to $f\in C_0[0,1]$, then 
	 	       \[|f(x)-f(y)|=\lim_{n\to\infty} |f_n(x)-f_n(y)| \leq  (1+\alpha)|x-y|.\]
	 	      Thus $f\in K_\alpha$. Therefore, $K_\alpha$ is a closed set.

\textbf{Uniform boundedness.} For any $f \in K_{\alpha}$,
	 	      \[ |f(x)|=|f(x)-f(0)| \leq (1+\alpha)x \leq 1+\alpha. \]
	 	       Hence $K_\alpha$ is uniformly bounded.
	 	       
\textbf{Equicontinuous.}	 For any $\varepsilon>0$, there exists $\delta=\frac{\varepsilon}{1+\alpha}$ such that 
	 	      \[|f(x)-f(y)| \leq (1+\alpha)|x-y| <\varepsilon\]
	 	      for all $|x-y|< \delta$ and $f\in K_\alpha$. Therefore, $K_\alpha$ is equicontinuous.

Therefore, we have shown that $K_{\alpha}$ is  compact in $C_0[0,1]$.
	 
 We now show the exponential tightness. By the definition of $\tilde{Z}_n(t)$, we know that 
	 \begin{equation*}
	 	\frac{d(\tilde{Z}_n(t))}{dt}=l(Y_{[nt]+1})-l(Y_{[nt]}),
	 \end{equation*}
	 for all $0\leq t \leq 1$ and $nt \notin \mathbb{N}$.  Thus $\tilde{Z}_n(\cdot) \in K_{\alpha}$ which implies $\tilde{\mu}_n(K_{\alpha}^C)=0$. Therefore,
	 \begin{equation*}
	 	\lim_{\alpha \to \infty} \limsup_{n\to\infty} \frac{1}{n} \log \tilde{\mu}_n(K_{\alpha}^C)=-\infty.
	 \end{equation*}
	 Thus the sequence $(\tilde{\mu}_n)$ is exponentially tight in $C_0([0,1])$.
\end{proof}

 With the groundwork laid in previous sections, we now complete the proof of our principal result, Theorem \ref{thm-ldp}.

\begin{proof}[Proof of Theorem \ref{thm-ldp}]
By Lemma \ref{lem-tildemu_n-ldp}, the probability measures sequence $( \tilde{\mu}_n )$ satisfies the LDP in the topology space $\mathcal{X}$ (endowed with the pointwise convergence topology) with the rate function $I_{\mathcal{X}}$.

From the definition of $\tilde{Z}_n(\cdot)$, we observe that the effective domain $\mathcal{D}_I \subset C_0([0,1])$ and  $\tilde {\mu}_n(C_0([0,1]))=1$.  Thus by  Proposition \ref{thm-Dembo-lem4.1.5}, the same LDP also holds in the space $(C_0([0,1]),\tau_2)$, where  $\tau_2$ is the pointwise convergence topology. 

Let $\tau_1$ be the topology of uniform convergence on $C_0([0,1])$.  Since $\tau_2$ is generated by the sets $V_{t,x,\delta} :=\{g\in C_0([0,1]):|g(t)-x|<\delta\}$, $t\in(0,1]$, $x\in \mathbb{R}$, $\delta>0$ and each $V_{t,x,\delta}$ is an open set in $\tau_1$, it follows that $\tau_1$ is finer than $\tau_2$, i.e., $\tau_2 \subset\tau_1$.
Hence, by the exponential tightness of $(\tilde{\mu}_n)$ in $(C_0([0,1]),\tau_1)$   (Lemma \ref{lem-tildemu_n-exptig}) and Lemma \ref{thm-Dembo-cor4.2.6},  we can lift the LDP of $(\tilde{\mu}_{n})$ to $(C_0([0,1]),\tau_1)$ with the same rate function $I_{\mathcal{X}}$.

Since $C_0([0,1])$ is a closed subset of $L_{\infty}([0,1])$, by Theorem \ref{thm-Dembo-lem4.1.5},  it follows that  $(\tilde{\mu}_n)$ satisfies the LDP in $(L_{\infty}([0,1]),\tau_1)$ and now the rate function is 
\begin{align*}
        I(f)=\begin{cases}
            I_{\mathcal{X}}(f), & f \in  C_0([0,1]),\\
            \infty, &  \text{otherwise}.
        \end{cases}
    \end{align*}
Therefore, since $(\mu_n)$ and $(\tilde{\mu}_n)$ are exponentially equivalent in $L_\infty([0,1])$ (Lemma \ref{lem-expequ}), then by Theorem \ref{thm-Dembo-thm4.2.13}, $(\mu_n)$ also satisfies the same LDP in $(L_{\infty}([0,1]),\tau_1)$ with the rate function 
\begin{align*}
        I(f)=\begin{cases}
            \sup_{J\in \mathcal{J}} I_J(p_J(f)), & f \in  C_0([0,1]),\\
            \infty, &  \text{otherwise}.
        \end{cases}
    \end{align*}

Finally, due to Lemma \ref{lem-mu_n^J-ratefun}, $\Lambda^J$ exists and is finite everywhere, and $I_J$ is the Fenchel-Legendre transform of $\Lambda^J$, namely,
\begin{equation*}
I_J(x) = \sup_{\lambda \in \mathbb{R}^j} \left\{ \langle \lambda, x \rangle - \Lambda^J(\lambda) \right\}, \quad \forall x \in (\mathbb{R}_{\geq 0})^j.
\end{equation*}
Now we complete the proof of the theorem.
\end{proof}

\section{Examples and Questions}

\subsection{Simple Random Walks on Regular Trees}
In this section we consider the simple random walk $(Y_n)_{n \geq 0}$ starting at $e$ on regular trees $T_d$, $d\geq3$, i.e., $Y_0=e$ and $Y_n =X_1 \cdots X_n$ for every $n \geq 1$, where the $X_n$ are independent and uniformly distributed random variables on $S=\{a_1,\ldots,a_d\}$ (the generating set of $T_d$).

It is easy to see that $l(Y_n)$ (the word length of $Y_n$) is a Markov chain on $\mathbb{N}=\{0 ,1, 2,\ldots \}$ with the following transition probabilities: for $u, v \in \mathbb{N}$ with $|u-v|=1$,  
\begin{equation}
p^Y(u,v) = \begin{cases}
 1 & \textrm{if $u=0, v=1$,}\\
 \frac{d-1}{d} & \textrm{if $u > 0, v=u+1$,}\\
 \frac{1}{d} & \textrm{if $u > 0, v=u-1$.}
  \end{cases} 
\end{equation}

For the Markov chain $l(Y_n)$, we can couple it with a biased random walk  $(R_n)_{n=0}^{\infty}$ on $\mathbb{Z}$ with the parameter $\frac{1}{d-1}$. The  biased random walk $(R_n)_{n=0}^{\infty}$ is a Markov chain on $\mathbb{Z}$ with the following transition probabilities: for $u, v \in \mathbb{Z}$ with $|u-v|=1$,
\begin{equation}
p^R(u,v) = \begin{cases}
\frac{1}{2} & \textrm{if $u=0$,}\\
 \frac{d-1}{d} & \textrm{if $|v| = |u|+1$,}\\
 \frac{1}{d} & \textrm{if $|v| = |u|-1$,}
\end{cases}
\end{equation}
in particular, 
\begin{equation}
	\begin{cases}
		p(u,u+1)= 1-p(u,u-1)=\frac{d-1}{d} & \textrm{if $u>0$,} \\
		p(u,u+1)= 1-p(u,u-1)=\frac{1}{d} & \textrm{if $u<0$,} \\
		p(0,1)=p(0,-1)=\frac{1}{2} & \textrm{if $u=0$.}
	\end{cases}
\end{equation}
Hence, $|R_n|$ ( the absolute value of $R_n$) is a Markov chain with the same transition probabilities as $l(Y_n)$. So the sample path LDP for $l(Y_n)$ is the same as that of $|R_n|$. Therefore, using the results of the biased random walk, we can get an explicit expression of the rate function of the sample path LDP for $l(Y_n)$.

First, due to the result of \cite[Theorem 3.1]{LSWX20}, we know that $\left(\frac{1}{n}|R_n|\right)$ satisfies the LDP with the good rate function given by 
\begin{equation} \label{Lambda*}
\Lambda^{*}(x)=\begin{cases}
	-\frac{1+x}{2} \log (d-1) +(1+x) \log\sqrt{1+x}+(1-x) \log\sqrt{1-x} + \log\frac{d}{2} & x\in[0,1], \\
	\infty & \textrm{otherwise.}
\end{cases}	
\end{equation}
Then adopting the strategy similar to the proof of \cite[Theorem 5.1.2]{DZ09}, the sample path LDP (Mogulskii type theorem) holds for $|R_n|$. In particular,  the law of 
\begin{equation*}
	Z_n'(t) :=\frac{1}{n} |R_{[nt]}|,  \ 0\leq t \leq 1,
\end{equation*}
in $L_{\infty}([0,1])$ satisfies the LDP with the good rate function
\begin{equation} \label{ratefunction}
	I(f) = \begin{cases}
		\int_0^1 \Lambda^{*}(\dot{f}(t)) dt & \textrm{if $f \in \mathcal{AC}_0$,} \\
		0 & \textrm{otherwises,}
	\end{cases}
\end{equation}
where $\Lambda^{*}$ is given by \eqref{Lambda*} and $\mathcal{AC}_0$ is the space of nonnegative absolutely continuous functions $f$ on $[0,1]$ starting at $0$ such that $\| f\| \leq 1$, i.e.
\begin{eqnarray*}
	\lefteqn{\mathcal{AC}_0 := \left\{  f\in C_0([0,1]):  \| f\| \leq 1, \sup_{0\leq s< t\leq 1} \frac{|f(t)-f(s)|}{t-s} \leq 1 \text{ and } \right. } \\
	& \left. \sum_{l=1}^k |t_l- s_l| \to 0, s_l < t_l \leq s_{l+1} < t_{s+1} \Rightarrow  \sum_{l=1}^k |f(t_l)- f(s_l)|\to 0 \right\}.
\end{eqnarray*}

Therefore, the rate function of $(\mu_n)$ in Theorem \ref{thm-ldp} in the case of the simple random walk is the same as $I$ given by \eqref{ratefunction}. For more detail, see \cite[Remark 3.1]{LSWX20}.

\subsection{More Questions}
There are still many shortcomings in our work. First, while we establish the sample path LDP, our analysis only provides an implicit expression of the rate function $I$ rather than its explicit form. This limitation stems from current theoretical challenges in calculating the rate function of the LDP for path-dependent Markov chains. 

Furthermore, only the convexity of the rate function is considered, and no more research is conducted on strict convexity, derivative properties etc. Therefore, we want to ask:
\begin{itemize}
	\item Can the rate function of the sample path LDP be calculated?
	\item Is the rate function $I$ in Theorem \ref{thm-ldp} always strictly convex?  What is the derivative of $I$ and what are its related properties
\end{itemize}

Based on the results of this paper, the following questions can also be considered:
\begin{itemize}
	\item Is it possible to calculate the probabilities that simultaneously account for the terminal letter of the reduced word and the word length by combining the calculation method of Lalley et al?
	\item Consider the sample path LDP on more general graphs.
\end{itemize}

 \section*{Acknowledgments}
 The authors would like to thank Longmin Wang for numerous valuable suggestions and comments to improve the quality of this paper. The project is supported partially by the National Natural Science Foundation of China (No. 12171252).

\bibliographystyle{alpha}
  \bibliography{SPLDP.bib} 
\end{document}